\documentclass[12pt,eps]{article}
\usepackage{hadronic}

\usepackage{graphicx}
\usepackage{epstopdf}
\usepackage[english]{babel}
\usepackage{graphicx,epstopdf,epsfig}
\usepackage{amsfonts,epsfig,fancyhdr,graphics, hyperref,amsmath,amssymb}

\newtheorem{remark}{Remark}
\newtheorem{example}{Example}
\newtheorem{conjecture}{Conjecture}
\newtheorem{proposition}{Proposition}
\newtheorem{corollary}{Corollary}
\newtheorem{theorem}{Theorem}
\newtheorem{definition}{Definition}
\newtheorem{proof}{Proof}
\newtheorem{keywords}{Keywords}
\newtheorem{AMS}{AMS}
\newtheorem{lemma}{Lemma}
\begin{document}
\thispagestyle{empty}
\noindent
\textbf{ALGEBRAS GROUPS AND GEOMETRIES 37} 193-237(2021)  \\
DOI: 10.29083/AGG.37.02.2021
\begin{center}\textbf{HALIDON RINGS AND THEIR APPLICATIONS} \end{center}
%\titlecomment{{\lsuper*}OPTIONAL comment concerning the title, \eg,
 % if a variant or an extended abstract of the paper has appeared elsewhere.}

\begin{center} A. Telveenus	\\%required
Kingston University London International Study Centre,\\ Stable Block, Kingston Hill Campus, KT2 7LB,UK \\	%required
email: t.fernandezantony@kingston.ac.uk \\ \end{center} %optional
\footnote{Former Prof. \& Head, Dept. of Mathematics,Fatima Mata National College, University of Kerala, Kollam, S. India . Supported by no grants from any sources.}	%optional

\begin{abstract}
 Halidon rings are rings with a unit element, containing a primitive $m^{th}$ root of unity and $m$ is invertible in the ring. The field of complex numbers is a halidon ring with any index $ m \geq 1$. This article examines different applications of halidon rings. The main application is the extension of Maschke's theorem. In representation theory, Maschke's theorem has an important role in studying the irreducible subrepresentations of a given group representation and this study is related to a finite field of characteristic which does not divide the order of the given finite group or the field of real or complex numbers. The second application is the computational aspects of halidon rings and group rings which enable us to verify  Maschke's theorem. Some computer codes have been developed to establish the existence of halidon rings which are not fields and the computation of units, involutions and idempotents in both halidon rings and halidon group rings. The third application of halidon rings is in solving Rososhek’s problem related to some cryptosystems. The applications in Bilinear forms and Discrete Fourier Transform (DFT) with two computer codes developed to calculate DFT and inverse DFT have also been discussed.
\end{abstract}

\begin{keywords}
Maschke's theorem; Primitive $m^{th}$ roots of unity;  halidon rings; Rososhek’s problem, bilinear forms;  Discrete Fourier Transform.
\end{keywords}
\begin{AMS}
16S34,20C05, 68W30
\end{AMS}

% Sample article for the Electronic Journal of Linear Algebra

%%%%%%%%%%%%%%%%%%%%%%%%%%%%%%%%%%%%%%%

%%%%%%%%%%%%%%%%%%%%%%%%%%%%%%%%%%%%%%%%%%%%%%%%%%%%%%%%%%%%%
\section{Introduction}
In 1940, the famous celebrated mathematician Graham Higman published a theorem \cite{gh} in group algebra which is valid only for a field or an integral domain with some specific conditions. In 1999, the author noticed that this theorem can be extended to a rich class of rings called halidon rings\cite{at}. In complex analysis, $|z|=1$ is the set of all points on a circle with centre at the origin and radius $1$. The solutions of $z^{n}=1$, which are usually called the $n^{th}$ roots of unity, which can be computed by DeMoivre's theorem, will form a regular polygon with vertices as the solutions. As $n\longrightarrow \infty $ this regular polygon will become a unit circle. This is the ordinary concept of $n^{th}$ of unity in the field of complex numbers. For a ring, we need to think differently. \\

A primitive $m^{th}$ root of unity in a ring with unit element is completely different from that of in a field, because of the presence of nonzero zero divisors.  So we need a separate definition for a primitive $m^{th}$ root of unity. An element $\omega $ in a ring $R$ is called a  \textit{primitive} $m^{th}$ root if $m$ is the least positive integer such that  $\omega^{m}=1$ and
\begin{eqnarray*}
\sum_{r=0}^{m-1} \omega^{r(i-j)}&=& m, \quad  i= j (\ mod \ m )\\ &=& 0, \quad  i\neq j (\ mod \ m ). \end{eqnarray*}
More explicitly,
\begin{eqnarray*}
1+ \omega^{r}+(\omega^{r})^{2}+(\omega^{r})^{3}+(\omega^{r})^{4}+......+(\omega^{r})^{m-1}&=& m, \quad  r=0 \\ &=& 0, \quad 0<r\leq m-1. \end{eqnarray*} \\

A ring $R$ with unity is called a \textit{halidon} ring with index $m$ if there is a  primitive $m^{th}$ root of unity and $m$ is invertible in $R$. The ring of integers is a halidon ring with index $ m=1$ and $\omega=1$. The halidon ring with index $1$ is usually called a \textit{trivial} halidon ring. The field of real numbers is a halidon ring with index $m=2$ and $\omega = -1$. The field $\mathbb{Q}$ $(i)=\{ a+ib | a,b \in$ $\mathbb{Q}$ $\}$ is a halidon ring with $\omega=i$ and $m=4$. $\mathbb{Z}_{65}$ is an example of a  halidon ring with index 4 and $\omega =8$ which is not a field. In general $\mathbb{Z}_{4r^{2}+1}$ is a halidon ring  with index 4 and $\omega =2r$ for each integer $r>0$. $\mathbb{Z}_{p}$ is a halidon ring with index $p-1$ for every prime $p$. Interestingly, $\mathbb{Z}_{p^{k}}$ is also a halidon ring with same index for any integer $k>0$ and it is not a field if $k>1.$ Note that if $\omega$ is a primitive $m^{th}$ root of unity, then $\omega^{-1}$ is also a primitive $m^{th}$ root of unity.
\section{Preliminary results}
Let $U(R)$ and $ZD(R)$ denote the unit group of $R$ and the set of zero divisors in $R$ respectively. Clearly, they are disjoint sets for a finite commutative halidon ring $R$.
\begin{lemma} \label{rd3}
A finite commutative ring $R$ with unity is a halidon ring with index $m$ if and only if there is a primitive $m^{th}$ root of unity  $\omega$ such that $m$, $\omega^{r}-1 \in U(R)$; the unit group of $R$ for all $r=1,.., m-1$. If $m$ is a prime number it is enough to have $m$, $\omega -1 \in U(R)$.
\end{lemma}
\begin{proof}
The proof follows from the fact that $(\omega^{r}-1)(1+ \omega^{r}+(\omega^{r})^{2}+(\omega^{r})^{3}+(\omega^{r})^{4}+......+(\omega^{r})^{m-1})=\omega^{rm}-1=0$.
\end{proof}
\begin{proposition} \label{rd15}
Let $R$ be a finite commutative halidon ring with index $m$. Then $R=U(R)\cup ZD(R)$ with $U(R)\cap ZD(R)= \{ \}$; the empty set.
\end{proposition}
\begin{proof}
Since $R$ is finite and $m$ is invertible in $R$, the proof is evident.
\end{proof}
The next theorem will give an improved necessary and sufficient conditions for a ring to be a halidon ring.
\begin{theorem}  \label{rd2}
A finite commutative ring $R$ with unity is a halidon ring with index $m$ if and only if there is a primitive $m^{th}$ root of unity  $\omega$ such that $m$, $\omega^{d}-1 \in U(R)$; the unit group of $R$ for all divisors $d$ of $m$ and $d<m$.
\end{theorem}
\begin{proof}
If R is a finite commutative halidon ring with index $m$, then it is evident that there is a primitive $m^{th}$ root of unity  $\omega$ such that $m$, $\omega^{d}-1 \in U(R)$; the unit group of $R$ for all divisor $d$ of $m$. \\
Conversely, assume that there is a primitive $m^{th}$ root of unity  $\omega$ such that $m$, $\omega^{d}-1 \in U(R)$; the unit group of $R$ for all divisor $d$ of $m$. We would like to show that $\omega^{r}-1 \in U(R)$, for $r=1,2,3,...,m-1$. Let $g$ be the greatest common divisor of $r$ and $m$. Then there are integers $u$ and $v$ such that $ur+vm=g$. Using the well known geometric series, we have $$(p-1)(1+p+p^{2}+...+p^{k-1})=p^{k}-1. $$ Put $p=\omega^{r}$ and $k=u$. Then we get $ (\omega^{r}-1).a=\omega^{ru}-1$ for some $a \in R$. This means that $(\omega^{r}-1)$ divides $\omega^{ru}-1=\omega^{ru}.\omega^{vm}-1=\omega^{ru+vm}-1=\omega^{g}-1$. Now we can choose a divisor $d$ of $m$ such that $g$ divides $d$. Let $p=\omega^{g}$ and $k=\dfrac{d}{g}$. Thus we have
$$ (\omega^{g}-1).b=\omega^{d}-1$$ for some $b \in R$. If $(\omega^{g}-1).c=0$, multiplying by $b$, we get $(\omega^{d}-1).c=0$. Since $(\omega^{d}-1) \in U(R)$, $c =0$. Therefore $(\omega^{g}-1)$ is not a zero divisor. Since $R$  is finite and using proposition \ref{rd15}, we get $(\omega^{g}-1)\in U(R)$. \\
Also, $(\omega^{r}-1)$ divides $\omega^{g}-1$. Therefore $(\omega^{r}-1)$ is a unit in $R$. Thus we have $(\omega^{r}-1)$ is a unit in $R$ for $r=1,2,,3..., m-1.$ By lemma \ref{rd3}, $R$ is a halidon ring with index $m$.
\end{proof}
The following results are stated without proof. For proofs, refer to \cite{at} and \cite{ath}.
\begin{proposition} Let R be a finite commutative halidon ring with index m. Then m divides the number of nonzero zero divisors in R.  \end{proposition}
\begin{proposition}
Let R be a commutative halidon ring with index m and let k $>$ 1 be a divisor of m. Then R is also a halidon ring with index k.
\end{proposition}

\begin{proposition}  \label{rd12}
Let I be an ideal of a halidon ring with index m. Then R/I is also a halidon with same index m.
\end{proposition}

The cardinality of a finite non-trivial commutative halidon ring is given by the following proposition.
\begin{proposition}
Let R be a finite commutative  halidon ring with index m. Then $|R|$=1 (mod m).
\end{proposition}
\begin{example}
Let $R=Z_{p^{k}}$ where p is an odd prime. Then $R$ is a halidon ring with index $m=p-1$ and $p^{k}=(1+p+p^{2}+p^{3}+...+p^{k-1})(p-1))+1=1 \ mod \ m$. $\therefore \ |R|=1 \ mod \ m$.
\end{example}

\begin{proposition}
Let R be a finite commutative halidon ring with index $m>1$ such that U(R)=$<\omega>$, where $<\omega>$ is the cyclic group generated by $\omega$; a primitive $m^{th}$ root of unity. Then \begin{enumerate}
                               \item $R=T\bigoplus L$, where T is a subfield of R and L is a subspace of R as a vector space over T,
                               \item R is semisimple.
                             \end{enumerate}
 \end{proposition}
The next theorem will give the necessary and sufficient conditions for a non-trivial halidon ring to become a filed.
 \begin{theorem} \label{rd5}
 A ring $R$ is a finite field if and only if $R$ is a finite commutative halidon ring with index $m>1$ such that $U(R)=<\omega>$;where $<\omega>$ is the cyclic group generated by $\omega$; a primitive $m^{th}$ root of unity .
 \end{theorem}
 \begin{example} Let $R=Z_{25}$. Then $U(R)=<2>$ and $2^{20}=1$. $R$ is not a halidon ring as the index $20$ is not invertible in $R$. However, $R=Z_{5}$ is a halidon ring with index $4$, $\omega=2$ and $U(R)=<2>$ which is also a field.
 \end{example}

 Let $$u= \sum_{i=1}^{m}\alpha_{i}g_{i}$$ be an element in the group algebra $RG$ and let $$\lambda_{r}=\sum_{i=1}^{m}\alpha_{m-i+2}(\omega^{(i-1)})^{(r-1)}$$ where $\omega \in R$ is a primitive $m^{th}$ root of unity. Then $u$ is said to be \textit{depending } on $\lambda_{1},\lambda_{2},......,\lambda_{m}$.

 \begin{proposition}
 Let S be a subring with unity of a halidon ring R with index m and let $G =<g>=\{g_{1}=1,g_{2}=g,......,g_{m}=g^{m-1}\}$ be a cyclic group of order $m$ generated  by $g$. Let $$u= \sum_{i=1}^{m}\alpha_{i}g_{i} \in U(RG)\cap SG$$ be depending on $\lambda_{1},\lambda_{2},......,\lambda_{m}$.Then $u^{-1} \in SG $ if and only of $\lambda=\lambda_{1}\lambda_{2}......\lambda_{m}$ is invertible in S.
 \end{proposition}
 The next theorem is about the existence of units in integral group rings which is an important area of study in group rings or group algebras.
 \begin{theorem}
 Let $\omega \in \mathbb{C} $ be a complex primitive $m^{th}$ root of unity and let $G =<g>=\{g_{1}=1,g_{2}=g,......,g_{m}=g^{m-1}\}$ be a cyclic group of order $m$ generated  by $g$. Let $$u= \sum_{i=1}^{m}\alpha_{i}g_{i} \in \mathbb{C}G$$ be depending on $\lambda_{1},\lambda_{2},......,\lambda_{m}$.
 Let $$\lambda_{i} ^{*}=\prod_{r=1,r \neq i}^{m} \lambda_{r} $$ and $$ \lambda=\prod_{r=1}^{m} \lambda_{r} $$ Then $$u= \sum_{i=1}^{m}\alpha_{i}g_{i} \in U(\mathbb{Z}G)$$ if and only if $ \lambda = \pm 1$. If $ \lambda_{m}^{-1}=f(\omega)$; a polynomial function of $ \omega$, then \\ $u^{-1}=f(g)$.
 \end{theorem}
 \begin{remark}
 When we use the formula $u^{-1}=f(g)$, ensure that the use of  $$ \sum_{r=0}^{m-1} \omega^{r}=0 $$ should be avoided in the calculation of $ \lambda_{i}$ or $ \lambda$ in order to make the behaviour of $ \omega $ exactly same as that of g. For an example see \cite{at}.
 \end{remark}

 \begin{theorem} \label{rd10}
 Let R be a commutative halidon ring with index m and let G be a cyclic group of order m. Then $RG\cong R^{m}$ as R-algebras.
  \end{theorem}
  \begin{theorem}[\textbf{Higman's Theorem} \cite{gh}, \cite{gk}] \label{rd13}
 Let $R$ be a commutative halidon ring with index $m$ and let $G$ be an finite abelian group of order $n$ with exponent $m$ and $n$ is a unit in $R$. Then $RG\cong R^{n}$ as $R$-algebras.
  \end{theorem}

\section{Main Result- Maschke's Theorem}
In this section, we discuss the characters of a group over halidon rings and Maschke's Theorem. Usually, the character of a group is defined over a real field $\mathbb{R}$ or a complex field $\mathbb{C}$. But, here we define character of a group over a halidon ring which need not be a field. Since the field of complex numbers is a halidon ring with any index greater than 1, all the properties of characters over halidon rings will automatically satisfied over the field of complex numbers. Now we are looking into those properties of characters of a group over a complex field which are also true over halidon rings. \\  \\ Let $R$ be a commutative halidon ring with index $m$ and let $G$ be a finite group of order $n$ with exponent $m$ such that $n$ is invertible in $R$. A homomorphism  $\rho : G\rightarrow GL(k,R)$ is defined as a \textit{representation} of $G$ over $R$ of \textit{degree} $k$, where $GL(k,R)$ is the general linear group of invertible matrices of order $k$ over $R$.  The \textit{character} of  G is defined as a homomorphism $\chi : G\rightarrow R $ such that
  $\chi (g)=tr(\rho(g))$ for $\forall \ g \in G $ where $tr(\rho(g)$ is the trace of the matrix $\rho(g)$. A representation is  \textit{faithful} if it is injective. When $k=1$, we call the character $\chi$ as  \textit{linear} character and it is actually a homomorphism $\chi : G\rightarrow u(R) $. In the standard definition, by a \textit{linear} character $\chi$ of G, mean a homomorphism $\chi: G\rightarrow \mathbb{C}-\{0\}$. If $R=\mathbb{C}$, then $U(R)=\mathbb{C}-\{0\}$. So the linear characters are precisely  $ \Gamma^{[r]}$.
   Thus we have the following theorem which is already known to be true for linear characters of a group over $\mathbb{C}$.
   \begin{theorem}
  Let R be a commutative halidon ring with index m and let G be a group of order n and exponent m such that n is invertible in R.  Let $G^{*}$ be the set of homomorphisms $\Gamma^{[s]}$ of G into U(R), which are given by(*). Then $G^{*}$ is a group under the product $\Gamma^{[s]}\Gamma^{[t]}=\Gamma^{[s+t]}$ and $G^{*}\cong G.$
  \end{theorem}
  \begin{proof} Refer to \cite{ath} for the proof.
  \end{proof}
 Suppose that $\rho : G\rightarrow GL(k,R)$ is  a \textit{representation} of $G$ over $R$. Write $V=R^{k}$, the module of all row vectors
 $(\lambda_{1},\lambda_{2},...,\lambda_{k})$ with each $\lambda_{i} \in R $. For all $v \in R$ and $ g \in G$, the matrix product $$ v \rho(g),$$ of the row vector $v$ with the matrix $\rho(g)$ is a row vector in $V$. With the multiplication $vg=v \rho (g)$, $V$ becomes an $RG$-module. Refer to \cite{gj}.\\
 The next theorem is the Maschke's Theorem for a cyclic group over a halidon ring.
 \begin{theorem} \label{rd11}
 Let G be a finite cyclic group of order m and let R be a commutative halidon ring with index m. Let $V=R^{m}$ be an RG-module. Then there are m irreducible RG-submodules $U_{1}, U_{_{2}},......., U_{m}$  such that $V=U_{1}\oplus U_{2} \oplus ...... \oplus U_{m}$.
 \end{theorem}
 \begin{proof}
 Let $\rho$ be the regular representation of $G=<g | g^{m}=1>$. Then $$ \rho (g)= cirulant (0,0,0,....,1)=\left(
                                                                                                         \begin{array}{ccccc}
                                                                                                           0 & 0 & 0 ....& 0 & 1\\
                                                                                                           1 & 0 & 0 ....& 0 & 0 \\
                                                                                                           . & . &  .....& . & . \\
                                                                                                           0 & 0 & 0 ....& 1 & 0 \\
                                                                                                         \end{array}
                                                                                                       \right)
 $$ is a square matrix of order $m$ and each row is obtained by  moving one position right and wrapped around of the  row above.  The eigen values are $$1, \ \omega, \  \omega^{2}\ ......, \omega^{m-1},$$ where $ \omega $ is a primitive  $m^{th}$ root of unity and with corresponding eigen vectors $(1, \ 1,...., \ 1),  (1, \ \omega, \  \omega^{2}\ ......, \omega^{m-1}),  (1, \ \omega^{2}, \  (\omega^{2})^{2}\ ......, (\omega^{m-1})^{2}), .....,$ \\ $(1, \ \omega^{m-1}, \  (\omega^{m-1})^{2}\ ......, (\omega^{m-1})^{m-1})$.
 \\ Let $U_{r}=span \{(1, \ \omega^{r-1}, \  (\omega^{r-1})^{2}\ ......, (\omega^{r-1})^{m-1}) \} $. Then any element $u_{r}$ of $U_{r}$ can be taken as $\lambda (1, \ \omega^{r-1}, \  (\omega^{r-1})^{2}\ ......, (\omega^{r-1})^{m-1})$ for some $\lambda \in R $. Therefore  \begin{eqnarray*} ug&=&u\rho(g) \\ &=& \lambda (1, \ \omega^{r-1}, \  (\omega^{r-1})^{2}\ ......, (\omega^{r-1})^{m-1})cirulant (0,0,0,....,1) \\ &=& \lambda\omega^{r-1}(1, \ \omega^{r-1}, \  (\omega^{r-1})^{2}\ ......, (\omega^{r-1})^{m-1}) \in U_{r}.  \end{eqnarray*} \\ This means that $U_{r}$ is an $RG$-submodule of $V$ for each $r=1,2,3,....m$. \\
 We define $\pi_{r}: V \rightarrow V $ by $$ \pi_{r}(v)=u_{r} \in U_{r}.$$ Then $\pi_{r}^{2}=\pi_{r}$ and therefore $V=U_{1}\oplus U_{2} \oplus ...... \oplus U_{m}$.
 \end{proof}
 \begin{corollary}
  Let $G$ be a finite cyclic group of order $k$ and let R be a \\ commutative halidon ring with index m. Let $V=R^{k}$ be an RG-module where   k $|$  m. Then there are k irreducible RG-submodules $U_{1}, U_{_{2}},......., U_{k}$  such that $V=U_{1}\oplus U_{2} \oplus ...... \oplus U_{k}$.
 \end{corollary}
 \begin{proof}
  Since  k $|$ m, we can take m=ck. Then $\omega_{1}=\omega^{c}$ is a primitive $k^{th}$ root of unity. Applying the above theorem, the result follows.
 \end{proof}
 \begin{theorem}[\textbf{Mascheke's Theorem}] \label{rd8}
Let G be a finite group of order n with exponent m, let R be a commutative halidon ring with index m and n is invertible in R. Let $V=R^{m}$ be an RG-module. If U is an RG-submodule of V, then there is an RG-submodule W such that $V=U\oplus W$.
 \end{theorem}
 \begin{proof}
  First choose any submodule $W_{0}$ of V such that $V=U\oplus W_{0}$. Since $V$ is a free module with rank $m$, we can find  a basis $ \{v_{1}, ..., v_{p}\}$ of U. Next we have to show that this basis can be extended to a basis $ \{ v_{1}, ..., v_{p}, v_{p+1},....,v_{m} \}$ of $V$.
Let $v_{p+1} \notin span\{v_{1}, ..., v_{p}\}$. Suppose that $$ \alpha_{1}v_{1}+\alpha_{2}v_{2}+.....+\alpha_{p}v_{p}+\alpha_{p+1}v_{p+1}=0.$$ Then

$$\alpha_{p+1}v_{p+1}=- \alpha_{1}v_{1}-\alpha_{2}v_{2}-.....-\alpha_{p}v_{p}$$
If $\alpha_{p+1}=0$, then $ v_{1}, ..., v_{p}, v_{p+1}$ are linearly independent. \\
If $\alpha_{p+1} \in U(R)$; the unit group of $R$, then $v_{p+1} \in span\{v_{1}, ..., v_{p}\}$, which is a contradiction. So $\alpha_{p+1} \notin U(R)$. \\
If $\alpha_{p+1} \neq 0 \in ZD(R)$; the set of zero divisors of $R$, then there exists \\ a $\beta_{p+1} \neq 0 \in ZD(R)$ such that $ \alpha_{p+1}\beta_{p+1}=0$. \\
 $$ \Rightarrow  \quad - \beta_{p+1}\alpha_{1}v_{1}-\beta_{p+1}\alpha_{2}v_{2}-.....-\beta_{p+1}\alpha_{p}v_{p}=0$$
 $ \Rightarrow \quad \beta_{p+1}\alpha_{i}=0  $ for  all  $i=1,2,3,...,p$.
Since the above is true for all $i=1,2,3,...,p$, $ \beta_{p+1}$ must be zero. This is a contradiction as $\beta_{p+1} \neq 0 $. Therefore $\alpha_{p+1}$ must be zero.
So $ v_{1}, ..., v_{p}, v_{p+1}$ are linearly independent. Continuing like this we can extend a basis $ \{v_{1}, ..., v_{p}\}$ of U to a basis $ \{ v_{1}, ..., v_{p}, v_{p+1},....,v_{m} \}$ of $V$.
Then $W_{0} =span \{v_{p+1}, ..., v_{m} \}$. \\ Rest of the proof is in line with \cite{gj}.
Now for all $v \in V$ there exist unique vectors $u \in U$ and $w \in W_{0}$ such that $v = u + w$.
We define $\phi : V \rightarrow V$ by setting $\phi (v)=u$. Recall from algebra that if $V=U \oplus W$, and if we
define $\pi: V \rightarrow V$ by
$$ \pi (u + w) = u \quad  for \ all \ u \ \in U, w \in W$$
then $\pi$ is an endomorphism of $V$. Moreover, $Im(\pi) =U$, $Ker(\pi) = W$ and $\pi^{2}=\pi$. With
this in mind we see that $\phi$ is a projection of $V$ with kernel $W_{0}$ and image $U$. Our aim is
to modify the projection $\phi$ to create an $RG$-homomorphism from $V \rightarrow V$ with image $U$.\\
Define $\tau : V \rightarrow V$ by
$$ \tau(v)=\frac{1}{n}\sum_{g \in G}g^{-1} \phi(gv), \quad v \in V $$
Then $ \tau $ is an endomorphism of $V$ and $Im(\tau ) \subseteq U$.
Now we will show that $ \tau $ is an $RG$-homomorphism. For $v \in V$ and $x \in G$ we have
$$ \tau(xv)=\frac{1}{n}\sum_{g \in G}g^{-1} \phi(gxv) $$
As g runs through the elements of G, so does $h = gx$. Thus we have \begin{eqnarray*}
 \tau(xv)&=&\frac{1}{n}\sum_{h \in G}xh^{-1} \phi(hv) \\
&=& \frac{1}{n}x \left(\sum_{h \in G}h^{-1} \phi(hv)\right) \\ &=& x \tau (v)\end{eqnarray*}

Thus $ \tau$ is an $RG$-homomorphism.

It remains to show that $\tau$ is a projection with image $U$. To show that $\tau$ is a projection
it suffices to demonstrate that $\tau^{2}=\tau$.
  Note that given $u \in U$ and $g \in G$, we have $gu \in U$,
so $\phi(gu)=gu$. Using this we see that:
\begin{eqnarray*}
\tau(u)&=&\frac{1}{n}\sum_{g \in G}g^{-1} \phi(gu) \\
&=&\frac{1}{n}\sum_{g \in G}g^{-1} gu \\
&=&\frac{1}{n}\sum_{g \in G}u \\
&=&u
\end{eqnarray*}
Now let $v\in V$. Then $\tau(v) \in U$, so we have $\tau(\tau(v))=\tau(v)$ . We have shown that $\tau^{2}=\tau$.

 \end{proof}
 In \cite{gj}, the authors have stated the Maschke's theorem for vector spaces over the field of real numbers $\mathbb{R}$ or complex numbers $\mathbb{C}$ and provided an example where Maschke's theorem can fail(see chapter 7) if the field is not a $\mathbb{R}$ or $\mathbb{C}$. In the light of the theorem \ref{rd8}, we can still have a vector space over the field $\mathbb{Z}_{p}$; which is a halidon ring with index $m=p-1$ for prime p. By theorem \ref{rd5}, there is a primitive $m^{th}$ root of unity $\omega$ such that $U(\mathbb{Z}_{p})=<\omega>$. Let $G=C_{m}=<a : a^{m}=1>$ and let $R=\mathbb{Z}_{p}$. Note that we cannot take $G$ as $C_{p}=<a : a^{p}=1>$ as $|G|=p$ has no multiplicative inverse in $\mathbb{Z}_{p}$. Let $V=R^{2}$ and let $\{ v_{1}, v_{2} \}$ be the standard basis for $V$. We define $\sigma: G \longrightarrow GL(2,R)$ by $\sigma(a^{j})= \left( \begin{array}{cc}
                                                                                                                              \omega^{j} & 0 \\
                                                                                                                              0 & \omega^{j}
                                                                                                                            \end{array}
  \right)$ for $j=1,2,3,....m.$. Clearly $\sigma$ is a representation of $G$. Also, $U=span \{ v_{1} \}$ is an $RG$-submodule of V. Define $W=span \{ v_{1}+v_{2}\}$, which is also an $RG$-submodule of $V$.
  Any element $v$ in $V$ can be written as $v=(\alpha_{1}-\alpha_{2})v_{1} + \alpha_{2}(v_{1}+v_{2})$ for some $\alpha_{1}, \alpha_{2} \in R$. If $x \in U\cap W$, then $x \in U$ and $x \in W$. So $x=\lambda_{1}v_{1}=(\lambda_{1},0)$ and $x=\lambda_{2}(v_{1}+v_{2})=(\lambda_{2}, \lambda_{2})$ for some $\lambda_{1}, \lambda_{2} \in R$. This implies $x=(0,0)$ and hence $V=U \oplus W$ as desired.
 \begin{definition}
 Let $R$ be a commutative halidon ring with index $m \ > 2 $ and let $S_{n}$ be the symmetric group on n symbols such that $|S_{n}|=m$. Let $S_{n}= \{ g_{1}, g_{2}, g_{3}, ..., g_{m}\}$ be in some order and let $ \{ e_{g_{1}}, e_{g_{2}}, e_{g_{3}}, ..., e_{g_{m}} \}$ be the standard basis for $V=R^{m}$. Define $\rho(g)(h)=e_{gh}$ for all $g, \ h \in S_{n}$. This is called the \textbf{permutation} \textbf{representation} of $S_{n}$.
 \end{definition}

\begin{example}
Let us order the elements of $S_{3}$ as follows: $$S_{3}= \{ g_{1}=id, \ g_{2}= (1,2) \ g_{3}= (1,3) \ g_{4}=(2,3) \ g_{5}=(1,2,3) \ g_{6}=(1,3,2)\}.$$ Then the composition table is given as below.

\begin{center}
\begin{tabular}{|c|c|c|c|c|c|c|}

         \hline
           & $g_{1}$ &  $g_{2}$ &  $g_{3}$ &  $g_{4}$ &  $g_{5}$ &  $g_{6}$ \\
            \hline
         $g_{1}$ &  $g_{1}$ &  $g_{2}$ & $g_{3}$ &  $g_{4}$ &  $g_{5}$ &  $g_{6}$ \\

         $g_{2}$ &  $g_{2}$ &  $g_{1}$ &  $g_{5}$ &  $g_{6}$ &  $g_{3}$ &  $g_{4}$ \\

         $g_{3}$ &  $g_{3}$ &  $g_{6}$ &  $g_{1}$ &  $g_{5}$ &  $g_{4}$ &  $g_{2}$ \\

         $g_{4}$ &  $g_{4}$ &  $g_{5}$ &  $g_{6}$ &  $g_{1}$ &  $g_{2}$ &  $g_{3}$ \\

         $g_{5}$ &  $g_{5}$ &  $g_{4}$ &  $g_{2}$ &  $g_{3}$ &  $g_{6}$ &  $g_{1}$ \\

        $g_{6}$ &  $g_{6}$ &  $g_{3}$ &  $g_{4}$ &  $g_{2}$ &  $g_{1}$ &  $g_{5}$ \\

        \hline

       \end{tabular}
\end{center}
Let $R$ be a halidon ring with index m=6 and $V=R^{6}$ (for example, $R=Z_{49}$). Let $ \{ e_{g_{1}}, e_{g_{2}}, e_{g_{3}}, ..., e_{g_{6}} \}$ be the standard basis for $V$. The representation $\rho$ is given by $\rho(g_{i})g_{j}=e_{g_{i}g_{j}}$. Using the composition table, for example, we can see that $$\rho(g_{3})=\left(
\begin{array}{cccccc}
0 & 0 & 1 & 0 & 0 & 0 \\
0 & 0 & 0 & 0 & 0 & 1 \\
1 & 0 & 0 & 0 & 0 & 0  \\
0 & 0 & 0 & 0 & 1 & 0 \\
0 & 0 & 0 & 1 & 0 & 0 \\
0 & 1 & 0 & 0 & 0 & 0 \\
\end{array}
\right)$$
Let $W_{0}=span \{ e_{g_{1}}, e_{g_{2}}, e_{g_{3}}, e_{g_{4}}, e_{g_{5}}\}$ and $U= span \{ e_{g_{1}}+ e_{g_{2}}+ e_{g_{3}}+e_{g_{4}}+e_{g_{5}}+ e_{g_{6}}\} $. Then clearly $V=W_{0}\oplus U$. We define a projection $\phi(e_{g_{i}})=0, for \ i=1,2,3,4,5$ and $\phi(e_{g_{6}})=e_{g_{1}}+ e_{g_{2}}+ e_{g_{3}}+e_{g_{4}}+e_{g_{5}}+ e_{g_{6}}.$ From the proof of Maschke's Theorem, we have $ \tau ( e_{g_{i}})=\frac{1}{6}( e_{g_{1}}+ e_{g_{2}}+ e_{g_{3}}+e_{g_{4}}+e_{g_{5}}+ e_{g_{6}} ) for \ all \ i=1,2,3,4,5,6$. Let $W=Ker \ \tau = span \{ e_{g_{1}}- e_{g_{2}}, e_{g_{2}}- e_{g_{3}}, e_{g_{3}}- e_{g_{4}}, e_{g_{4}}- e_{g_{5}}, e_{g_{5}}- e_{g_{6}}\}$. Then $wg \in W$ for all $w \in W $ and $g \in S_{3}$. Therefore $W$ is an $RS_{3}$ submodule and $V=U\oplus W$ as expected from the Maschke's Theorem.
\end{example}
\section{Rososhek’s problem}
The Rososhek's problem is a problem related to a cryptosystem using group rings \cite{knp}. Let R be a commutative ring and G be a group. We say that an automorphism $\psi : RG \longrightarrow RG$ is \textit{standard} if it is defined by some ring automorphism $\alpha : R \longrightarrow R$ and by a group automorphism $\sigma : G \longrightarrow G$, so that  $$\psi(x) = \sum_{g \in G}\alpha (a_{g})\sigma(g)$$ for every $$x =\sum_{g \in G}a_{g}g \in RG$$. \\
\textbf{Rososhek’s problem}: For which finite commutative rings R and finite groups G does the group ring
$RG$ have standard automorphisms only? \\
We refer to group rings $RG$ possessing just standard automorphisms as \textit{automorphically rigid}, and the problem then consists in defining
finite \textit{automorphically rigid} group rings.
In \cite{knp}, the author has proved the following theorem for a finite field:
\begin{theorem}
Let $K$ be a finite field of characteristic $p$ and G a finite group. A group algebra $KG$is
automorphically rigid if and only if one of the following conditions holds: \\ \begin{enumerate}
\item G is a trivial group, i.e., G = e;
\item G is a cyclic group of odd prime order q, i.e., $G = C_{q}$, q = 2, and $K = F_{p}$ is a prime field of characteristic
p, with p a primitive root modulo q;
\item G is a direct product of an order 2 cyclic group and a cyclic group $C_{q}$ of odd prime order q = 2, i.e.,
 $G = C_{2} \times C_{q}$, with 2 a primitive root modulo q, and $K = F_{2}$ is a two-element field;
\item $G$ is a permutation group on three symbols, i.e., $G = S_{3}$, and $K = F_{2}$ is a two-element field.
\end{enumerate}
\end{theorem}
Still, the question remains unanswered for a finite commutative ring which is not a field.
\begin{proposition} \label{rd14}
Let $R=\dfrac{\mathbb{Z}_{n}[X]}{ (X^{2}-1)}$. Then the number of automorphisms of $R$, $|Aut R|$ is given by\\
\begin{eqnarray*}
 |Aut R| &=& 2 \ if \ n=p^{s} \ where \ s \geq 1 \ is \ an \ integer \ and \ p \ is \ an \ odd \ prime \ number \\
 &=&2^{k} \ if  \ n=p_{1}^{e_{1}}p_{2}^{e_{2}}p_{3}^{e_{3}}.....p_{k}^{e_{k}} \ with \ prime \ numbers \ 2<p_{1}<p_{2}<.....<p_{k}.  \\
 \end{eqnarray*}
 \end{proposition}
 \begin{proof}
This proof is somewhat similar to example 2 in \cite{nm}.  Let $x$ be the image of $X$ in $R$. Any element in $R$ can be uniquely written as $ax+b$ with $a,b \in \mathbb{Z}_{n}$. Let $\sigma \in AutR$.
 Then $\sigma(a)=a$ for every $a \in \mathbb{Z}_{n}$. Therefore $\sigma(x)=ax+b$ for some $a, \ b \in \mathbb{Z}_{n} $. Since $\sigma $ is an automorphism, there exists an element $px+q$ with $p, \ q \in \mathbb{Z}_{n}$ such that $x=\sigma(px+q)=p\sigma(x)+q=pax+pb+q$. then we get $pa=1$ and so $a$ must be a unit in $\mathbb{Z}_{n}$. Further, if $\sigma(x)=ax+b$ with $a \in U(\mathbb{Z}_{n})$, we must also have
 $0=\sigma(x^{2}-1)=\sigma(x^{2})-1=(ax+b)^{2}-1=a^{2}x^{2}+2abx+b^{2}-1=a^{2}(x^{2}-1)+2abx+a^{2}+b^{2}-1=2abx+a^{2}+b^{2}-1$. \\
 Since $n$ is odd, $2ab=0\Longrightarrow b=0$ as $a \in U(\mathbb{Z}_{n})$. This implies $ a^{2}=1. $ Therefore $|AutR|$=no. of involutions in $\mathbb{Z}_{n}$. But, the number of involutions in $\mathbb{Z}_{n}$ clearly follows the statement in the proposition. Programme-3 mentioned below is useful to compute involutions in $\mathbb{Z}_{n}$.
 \end{proof}
 If $\mathbb{Z}_{n}$ is halidon ring with index $m$, then so is $\mathbb{Z}_{n}[X]$. By proposition \ref{rd12}, $\dfrac{\mathbb{Z}_{n}[X]}{ (X^{2}-1)}$ is also a halidon ring with index $m$.
\begin{proposition}
Let $R=\dfrac{\mathbb{Z}_{n}[X]}{ (X^{2}-1)}$ be a halidon ring with index $m$ and let $G=C_{m}$ be a cyclic group of order $m$. If $m!=2^{k} \times \phi(m)$ for some positive integer $k$, then $RG$ is automorphically rigid.
\end{proposition}
\begin{proof}
Clearly $R$ is not a field. By theorem \ref{rd13}, $RG\cong R^{m}$. So $|AutRG|=|AutR^{m}|$. But, $AutR^{m}$ is just the collection of automorphisms which permute the positions of the elements of $R^{m}$. Thus we have $|AutR^{m}|=m!$. By proposition \ref{rd14}, we have $|AutR|=2^{k}$ for some positive integer $k$. So $m!=2^{k} \times \phi(m)$ implies $|AutRG|=|AugR|\times |AugG|$ and therefore $RG$ is automorphically rigid.
\end{proof}
The only solution to the above proposition is $m=2$. So one of the solutions (there may have other solutions) to the Rososhek’s problem is given by $R=\dfrac{\mathbb{Z}_{p^{s}}[X]}{ (X^{2}-1)}$ where $p$ is an odd prime, $s \geq 1$ is an integer and $\omega=p^{s}-1=-1mod p^{s} \in \mathbb{Z}_{p^{s}}$ is a primitive $m^{th}$ root of unity and $G=C_{2}$.
\section{The computational aspects of halidon rings and halidon group rings}
The main purpose this section is to verify Maschke's Theorem using some computer codes. The computer programme-5, is very useful to verify the Maschke's theorem for cyclic group. Throughout this section, let $R=\mathbb{Z}_{n}$ be a halidon ring with index $m$ and primitive $m^{th}$ root of unity $\omega$. Let $G =<g>=\{g_{1}=1,g_{2}=g,......,g_{m}=g^{m-1}\}$ be a cyclic group of order $m$ generated  by $g$. We study the computational aspects of finding halidon rings for any integer $n>2$ and computing the units and idempotents in the halidon group ring $RG$ based on the related theorems.
\begin{definition} \label{rd7}
Let $p_{1},p_{2},p_{3},....,p_{k}$ be odd primes and let $\phi(x)$ be the Euler's totient function. We define the \textit{halidon function}  $$\psi(n)= \begin{cases} gcd \{ \phi(p_{1}^{e_{1}}), \phi(p_{2}^{e_{2}}),\phi(p_{3}^{e_{3}}),...., \phi(p_{k}^{e_{k}})\}, & n=p_{1}^{e_{1}}p_{2}^{e_{2}}p_{3}^{e_{3}}.....p_{k}^{e_{k}} \\
1, & n \ \text{is even}\end{cases} $$
\end{definition}
\begin{proposition}
Let $n$ be as in \ref{rd7}. Then the halidon function $$\psi(n)=gcd\{ p_{1}-1,p_{2}-1,p_{3}-1,....,p_{k}-1\},$$ which is independent of the exponents $e_{1},e_{2},e_{3},....,e_{k}$.
\begin{proof}
The proof follows immediately from the fact that $p_{1},p_{2},p_{3},....,p_{k}$ are distinct primes.
\end{proof}
\end{proposition}
It is well-known that the \textit{carmichael function} $\lambda(n)$ is the exponent of $U(\mathbb{Z}_{n})$. The proof of the following proposition is evident.
\begin{proposition}
If $\lambda(n)$ is the carmichael function, then
\begin{enumerate}
  \item $\psi(n)$ divides $\lambda(n)$,
  \item $\psi(n^{k})=\psi(n)$ if n is odd,
    \item $\psi(p_{1}p_{2}p_{3}....p_{s})=\psi(p_{1}^{d_{1}}p_{2}^{d_{2}}p_{3}^{d_{3}}.....p_{s}^{d_{s}})$ if each integer $d_{i}>0$ for $i=1,2,3,....,s$.
\end{enumerate}
\end{proposition}
I have developed $5$ computer programme codes in c++ (Microsoft Visual Studio 2019) based on the theorems \ref{rd2}, \ref{rd9} and $2$ programmes for Discrete Fourier Transforms. The programme codes are included in the appendix.

The computer programme-1 can be used to find a halidon ring $Z_{n}$ of given order $n\geq 1$. The author would like to provide its algorithm as follows: \\
***************************** \\
ALGORITHM for Programme-1 \\
***************************** \\

Input: n.\\
Output:$Z_{n}$ is a trivial halidon ring or not. \\
1. If n is even, then $Z_{n}$ is a trivial halidon ring.\\
2. Else \\
for i=1,2,...., n-1\\
for j=1,2,3,,,,,,n-1 do compute $i*j mod n$\\
if  i*j mod n=1, $w\longleftarrow i$\\
3. for i=1,2,3,,,,,, n-1 do compute $w^{i}$\\
4. if $ w^{i} \ mod n=1 \ $, $ m\longleftarrow i$ \\
5. compute divisors d of m and $d < m$ \\
6. compute $w^{d}-1$ and if $w^{d}-1=1modn$ for all d, then $Z_{n}$ is a nontrivial halidon ring with index m. \\
End

 There are infinitely many halidon rings which are not fields. For example, using the programme-1, we can see that: \begin{enumerate}
                                                          \item $\mathbb{Z}_{49}$ is halidon ring  with index $m=6$ and $\omega=19$,
                                                          \item $\mathbb{Z}_{2001}$ is a trivial halidon ring with index $m=2$ and $\omega=2000$,
                                                          \item $\mathbb{Z}_{2501}$ is halidon ring with index $m=20$ and $\omega=8$ or $2493$,
                                                           \item $\mathbb{Z}_{3601}$ is halidon ring with index $m=12$ and $\omega=1350$ or $2528$,
                                                          \item $\mathbb{Z}_{10001}$ is halidon ring with index $m=8$ and $\omega=10$ or $9220$,
                                                           \item $\mathbb{Z}_{100001}$ is halidon ring with index $m=10$ and $\omega=26364$ or $73728$ (running time 35 minutes).
                                                        \end{enumerate}
By running the same programme several times for different values of $n$, the author has come to the conclusion of the following conjecture.
 \begin{conjecture}
If $R=\mathbb{Z}_{n}$ and $n=p_{1}^{e_{1}}p_{2}^{e_{2}}p_{3}^{e_{3}}.....p_{k}^{e_{k}}$ with primes $p_{1}<p_{2}<p_{3}<....<p_{k}$ including 2, then $R$ is a halidon ring with maximal index $m_{max}=\psi(n)$.
\end{conjecture}

 \begin{theorem}\label{rd4}
 Let $$u= \sum_{i=1}^{m}\alpha_{i}g_{i} \in U(RG)$$ be depending on $\lambda_{1},\lambda_{2},......,\lambda_{m}$. Let $$v=\sum_{i=1}^{m}\beta_{i}g_{i}$$
 be the multiplicative inverse of $u$ in $RG$. Then $$\beta_{i}=\frac{1}{m}\sum_{r=1}^{m} \lambda_{r}^{-1}(\omega^{i-1})^{r-1}.$$
 \end{theorem}
The computer programme-2 can be used to test whether a given element $u$ in $RG$ is a unit or not. If it is a unit, then the programme will give the multiplicative inverse $v$ in $RG$. \\
\textbf{Input}: $n=121,m=10,m^{-1}=109,\omega =94, a[1]=62, a[2]=21, a[3]=22, a[4]=85, a[5]=81, a[6]=95, a[7]=24,a[8]=30, a[9]=1, a[10]=65$ \\
\textbf{Output}: The multiplicative inverse of $a=62+21g+22g^{2}+85g^{3}+81g^{4}+95g^{5}+24g^{6}+30g^{7}+g^{8}+65g^{9} $ is $b=102+68g+34g^{2}+61g^{3}+73g^{4}+54g^{5}+102g^{6}+109g^{7}+18g^{8}+455g^{9} $.\\
\textbf{Input}: $n=121,m=10,m^{-1}=109,\omega =94, a[1]=72, a[2]=71, a[3]=89, a[4]=48, a[5]=54, a[6]=0, a[7]=2,a[8]=105, a[9]=25, a[10]=19$ \\
\textbf{Output}: The multiplicative inverse of $a=72+71g+89g^{2}+48g^{3}+54g^{4}+0g^{5}+2g^{6}+105g^{7}+25g^{8}+19g^{9} $ is $b=72+71g+89g^{2}+48g^{3}+54g^{4}+0g^{5}+2g^{6}+105g^{7}+25g^{8}+19g^{9} $, which is an involution.\\
\textbf{Input}: $n=121,m=10,m^{-1}=109,\omega =94, a[1]=5, a[2]=7, a[3]=2, a[4]=40, a[5]=22, a[6]=90, a[7]=20,a[8]=25, a[9]=10, a[10]=55$ \\
\textbf{Output}: The element $a=5+7g+2g^{2}+40g^{3}+22g^{4}+90g^{5}+20g^{6}+25g^{7}+10g^{8}+56g^{9} $ has no multiplicative inverse. \\
A direct calculation shows that all outputs are correct.
 \begin{theorem}\label{rd4}
 Let $$u= \sum_{i=1}^{m}\alpha_{i}g_{i} \in RG$$ be depending on $\lambda_{1},\lambda_{2},......,\lambda_{m}$.
 Then \begin{enumerate} \label{rd9}
   \item $u \in U(RG)$ if and only if each $\lambda_{i} \in U(R)$,
  \item  $u \in E(RG)$ if and only if each $\lambda_{i} \in E(R)$, where $E(RG)$ is the set of idempotents in $RG$.
   \end{enumerate}
   More over, $|U(RG)|=|U(R)|^{|G|}$ and $|E(RG)|=|E(R)|^{|G|}$.
 \end{theorem}
 \begin{proof}
 The proof follows from the isomorphism $\rho(u)=(\lambda_{1},\lambda_{2},......,\lambda_{m})$ from $RG$ onto $R^{m}$. This is the isomorphism used to prove theorem \ref{rd10}.
 \end{proof}
 In order to find a unit element or an involution or an idempotent in $R$, we can use the programme-3.

\textbf{Input}: n=25 \\
\textbf{Output}: Involutions are $1$ and $24$. \\
Idempotents are $0$ and $1$. \\
The units and their inverses are $(1,1), \ (2,13), \ (3,17), \ (4,19), \ (6,21), \ (7,18), \\ \ (8,22), \ (9,14), \ (11,16), \ (12,23), \
(13,2), \ (14,10), \ (16,11), \ (17,3), \ (18,7), \\ \ (9,4), \ (21,5), \ (22,20), \ (23,12)$ and $(24,24)$. \\

 After finding the units and involutions in $R$, using the programme-4, we can compute the units and involutions in $RG$.  \\
\textbf{Input}: $n=25, m=4, m^{-1}=19, \omega =7, l[1]=7, l[2]=3, l[3]=13, l[4]=21, l1[1]=18, l1[2]=17, l1[3]=2, l1[4]=5$ \\
\textbf{Output}: The multiplicative inverse of $a=11+17g+24g^{2}+5g^{3}$ is $b=23+12g^{2}+8g^{3} $. \\
\textbf{Input}: $n=25, m=4, m^{-1}=19, \omega =7, l[1]=1, l[2]=24, l[3]=24, l[4]=1, l1[1]=1, l1[2]=24, l1[3]=24, l1[4]=1$ \\
\textbf{Output}: The multiplicative inverse of $a=22g+4g^{3}$ is $b=22g+4g^{3} $, which is an involution. \\
All outputs can be verified through direct calculations. \\
 After finding the idempotents in $R$, using the programme-5, we can compute the idempotents in $RG$.

\textbf{Input}: $n=49, m=6,m^{-1}=41, \omega=19,l[1]=1,l[2]=0,l[3]=0,l[4]=0,l[5]=0,l[6]=0$ \\
\textbf{Output}: $e_{1}=41+41g+41g^{2}+41g^{3}+41g^{4}+41g^{5}$ \\
\textbf{Input}: $n=49, m=6,m^{-1}=41, \omega=19,l[1]=0,l[2]=1,l[3]=0,l[4]=0,l[5]=0,l[6]=0$ \\
\textbf{Output}: $e_{2}=41+44g+3g^{2}+8g^{3}+5g^{4}+46g^{5}$ \\
\textbf{Input}: $n=49, m=6,m^{-1}=41, \omega=19,l[1]=0,l[2]=0,l[3]=1,l[4]=0,l[5]=0,l[6]=0$ \\
\textbf{Outpu}t: $e_{3}=41+3g+5g^{2}+41g^{3}+3g^{4}+5g^{5}$ \\
\textbf{Input}: $n=49, m=6,m^{-1}=41, \omega=19,l[1]=0,l[2]=0,l[3]=0,l[4]=1,l[5]=0,l[6]=0$ \\
\textbf{Output}: $e_{4}=41+8g+41g^{2}+8g^{3}+41g^{4}+8g^{5}$ \\
\textbf{Input}: $n=49, m=6,m^{-1}=41, \omega=19,l[1]=0,l[2]=0,l[3]=0,l[4]=0,l[5]=1,l[6]=0$ \\
\textbf{Output}: $e_{5}=41+5g+3g^{2}+41g^{3}+5g^{4}+3g^{5}$ \\
\textbf{Input}: $n=49, m=6,m^{-1}=41, \omega=19,l[1]=0,l[2]=0,l[3]=0,l[4]=0,l[5]=0,l[6]=1$ \\
\textbf{Output}: $e_{6}=41+46g+5g^{2}+8g^{3}+3g^{4}+44g^{5}$ \\
Clearly $e_{1},e_{2},e_{3},e_{4},e_{5}$ and $e_{6}$ are orthogonal idempotents such that $e_{1}+e_{2}+e_{3}+e_{4}+e_{5}+e_{6}=1$.
Therefore $\mathbb{Z}_{49}G=U_{1}\oplus U_{2} \oplus U_{3} \oplus U_{4} \oplus U_{5} \oplus U_{6}$ where  $U_{1}=span\{41+41g+41g^{2}+41g^{3}+41g^{4}+41g^{5}\} =span\{1+g+g^{2}+g^{3}+g^{4}+g^{5}\}$, \\ $U_{2}=span\{41+44g+3g^{2}+8g^{3}+5g^{4}+46g^{5}\} =span\{1+\omega g+\omega ^{2}g^{2}+\omega^{3}g^{3}+\omega^{4}g^{4}+\omega^{5}g^{5}\}$, \\$U_{3}=span\{41+3g+5g^{2}+41g^{3}+3g^{4}+5g^{5}\} =span\{1+(\omega)^{2} g+(\omega ^{2})^{2}g^{2}+(\omega^{2})^{3}g^{3}+(\omega^{2})^{4}g^{4}+(\omega^{2})^{5}g^{5}\}$, \\ $U_{4}=span\{41+8g+41g^{2}+8g^{3}+41g^{4}+8g^{5}\}  =span\{1+(\omega)^{3} g+(\omega ^{3})^{2}g^{2}+(\omega^{3})^{3}g^{3}+(\omega^{3})^{4}g^{4}+(\omega^{3})^{5}g^{5}\}$, \\$U_{5}=span\{41+5g+3g^{2}+41g^{3}+5g^{4}+3g^{5}\} =span\{1+(\omega)^{4} g+(\omega ^{4})^{2}g^{2}+(\omega^{4})^{3}g^{3}+(\omega^{4})^{4}g^{4}+(\omega^{4})^{5}g^{5}\}$, \\ $U_{6}=span\{41+46g+5g^{2}+8g^{3}+3g^{4}+44g^{5}\}=span\{1+(\omega)^{5} g+(\omega ^{5})^{2}g^{2}+(\omega^{5})^{3}g^{3}+(\omega^{5})^{4}g^{4}+(\omega^{5})^{5}g^{5}\}$, after multiplying by $6$, the inverse of $41$. This confirms the theorem \ref{rd11}. \\
\textbf{Input}: $n=65, m=4,m^{-1}=49, \omega=8,l[1]=1,l[2]=26,l[3]=40,l[4]=26$ \\
\textbf{Output}: The idempotent in $\mathbb{Z}_{65}G$ is $e=7+39g+46g^{2}+39g^{3}$.\\
A direct calculation verifies that all outputs are correct.

\section{Bilinear Forms and Circulant Matrices}
Let ring $R$ be a commutative halidon ring with index $m$ and primitive $m^{th}$ root of unity $\omega$. Let $G$ be a cyclic group of order $m$, generated by $g$. We define $ g_{i}=g^{i-1} $.  $ \therefore g_{i}g_{j}=g_{i+j-1}$; $i=1,2,3,....,m$. By the extension theorem of Higman, we have $RG\cong R^{m}$ as $R$-algebras and the isomorphism $\rho$ is given by $$\rho\left(\sum_{i=1}^{m}\alpha_{i}g_{i}\right)=(\lambda_{1},\lambda_{2},\lambda_{3}....,\lambda_{m})$$ where $$\lambda_{i}=\sum_{r=1}^{m}\alpha_{m-r+1}(\omega^{i-1})^{r-1}.$$ Since $\{g_{i}\}$ is an $R$-basis for $RG$, $\{\rho(g_{i})\}$ is a basis for $R^{m}$ and \\
$s_{1}=\rho(g_{1})=(1,1,1,...,1)$, \\ $s_{2}=\rho(g_{2})=(1,\omega^{m-1},(\omega^{m-1})^{2},...,(\omega^{m-1})^{m-1})$,\\ $s_{3}=\rho(g_{3})=(1,\omega^{m-2},(\omega^{m-2})^{2},...,(\omega^{m-2})^{m-1})$,\\.............,\\
$s_{m}=\rho(g_{m})=(1,\omega,\omega^{2},...,\omega^{m-1})$. \\
Let $\{e_{i}\}$ be the standard basis in $R^{m}$. Then \\
$s_{1}=e_{1}+e_{2}+e_{3}+...+e_{m}$, \\ $s_{2}=e_{1}+\omega^{m-1}e_{2}+(\omega^{m-1})^{2}e_{3}+...+(\omega^{m-1})^{m-1})e_{m}$,\\ $s_{3}=e_{1}+\omega^{m-2}e_{2}+(\omega^{m-2})^{2}e_{3}+...+(\omega^{m-2})^{m-1})e_{m}$,\\......................................,\\
$s_{m}=e_{1}+\omega e_{2}+\omega^{2}e_{3}+...+\omega^{m-1})e_{m}$. \\
This gives

$\left(
  \begin{array}{c}
    s_{1} \\ s_{2} \\s_{3} \\\dots \\ s_{m}
  \end{array}
\right)$ =
$\left(
  \begin{array}{ccccc}
   1 & 1 & 1 & \dots & 1 \\
1 & \omega^{m-1} & (\omega^{m-1})^{2} & \dots & (\omega^{m-1})^{m-1} \\
1 & \omega^{m-2} & (\omega^{m-2})^{2} & \dots & (\omega^{m-2})^{m-1} \\
\dots  & \dots  & \dots  & \dots & \dots  \\
1 & \omega & (\omega)^{2} & \dots & (\omega^{m-1}) \\
  \end{array}
\right)$
$\left(
  \begin{array}{c}
    e_{1} \\ e_{2} \\e_{3}\\ \dots \\ e_{m}
  \end{array}
\right)$

$$s^{T}=\Phi^{*}e^{T}$$
$$\therefore \quad \quad e^{T}=\frac{1}{m}\Phi s^{T}$$
where $$\Phi=\left(
                                                                        \begin{array}{ccccc}
                                                                          1 & 1 & 1 & ..... & 1 \\
                                                                          1 & \omega & \omega^{2} & ..... & \omega^{m-1} \\
                                                                          1 & \omega^{2} &(\omega^{2})^{2} & ..... & (\omega^{2})^{m-1}\\
                                                                          . & . & . & ..... & . \\
                                                                          . & . & . & ..... & . \\
                                                                          . & . & . & ..... & . \\
                                                                          1 & \omega^{m-1} & (\omega^{m-1})^{2} & ..... & (\omega^{m-1})^{m-1}\\
                                                                        \end{array}
                                                                      \right)$$
and $\Phi^{*}$ is the $\Phi$ conjugate transposed\cite{pjd}. Thus we have the following theorem.
\begin{theorem} \label{r1}
Let ring $R$ be commutative halidon ring with index $m$ and primitive $m^{th}$ root of unity $\omega$. Let $G$ be a cyclic group of order $m$, generated by $g$. We define $ g_{i}=g^{i-1} $  ; $i=1,2,3,....,m$ and let $\{ s_{i}\}$ be the image of $\{ g_{i}\}$ under the isomorphism $RG\cong R^{m}$ and let $\{ e_{i}\}$ be the standard basis for $R^{n}$. Then $s^{T}=\Phi^{*}e^{T}$ or $ e^{T}=\frac{1}{m}\Phi s^{T}$.
\end{theorem}

For each $u=(u_{1},u_{2},u_{3},...,u_{m})\in R^{m}$, we define $C_{u}=circu(u_{1},u_{2},u_{3}...,u_{m})$.
Also, we define $f_{u}: R^{m} \times R^{m}\longrightarrow R$ by $$f_{u}(x,y)=<x,y>_{u}=xC_{u}y^{T}$$ for every $x=(x_{1},x_{2},x_{3}...,x_{m}),y=(y_{1},y_{2},y_{3}...,y_{m}) \in R^{m}$. We adopt some standard definitions of bilinear form for $<x,y>_{u}$.   $<x,y>_{u}$ is said to $symmetric$ if $<x,y>_{u}=<y,x>_{u}$ for every $x,y \in R^{m}$. It is said to be $skewsymmetric$ if $<x,y>_{u}=-<y,x>_{u}$ for every $x,y \in R^{m}$. $<x,y>_{u}$ is said to $alternating$ if $<x,x>_{u}=0$ for every $x \in R^{m}$ \cite{sl}.
\begin{theorem} \label{r2}
$f_{u}$ is a bilinear form and
\begin{eqnarray*}  <s_{i},s_{j}>_{u}&=&m(u_{1}+u_{2}\omega^{i-1}+u_{3}(\omega^{i-1})^{2}+...+u_{m}(\omega^{i-1})^{m-1}), \\
  \text{if} \ i+j=2 \ (mod \ m) \\ &=&0 \quad \quad  \text{otherwise}. \end{eqnarray*}
\end{theorem}
\begin{proof}
It is clear that $f_{u}$ is a bilinear form.   \begin{eqnarray*}<e_{i},e_{j}>_{u}=e_{i}C_{u}e_{j}^{T}&=&u_{j-i+1} \quad \text{if}\quad i\leq j \\ &=&u_{m+j-i+1} \quad  \text{if} \quad i >j . \end{eqnarray*}$ <s_{i},s_{j}>_{u}=<e_{1}+\omega^{m-i+1}e_{2}+(\omega^{m-i+1})^{2}e_{3}+...+(\omega^{m-i+1})^{m-1})e_{m}, \\ e_{1}+\omega^{m-j+1}e_{2}+(\omega^{m-j+1})^{2}e_{3}+...+(\omega^{m-j+1})^{m-1})e_{m}>_{u} \\ = (u_{1}+\omega^{m-j+1}u_{2}+(\omega^{m-j+1})^{2}u_{3}+...+(\omega^{m-j+1})^{m-1})u_{m})\\+
(\omega^{m-i+1})(u_{m}+\omega^{m-j+1}u_{1}+(\omega^{m-j+1})^{2}u_{2}+...+(\omega^{m-j+1})^{m-1})u_{m-1})\\
+(\omega^{m-i+1})^{2}(u_{m-1}+\omega^{m-j+1}u_{m}+(\omega^{m-j+1})^{2}u_{1}+...+(\omega^{m-j+1})^{m-2})u_{m-1}) \\ +...
+(\omega^{m-i+1})^{m-1}(u_{2}+\omega^{m-j+1}u_{3}+(\omega^{m-j+1})^{2}u_{4}+...+(\omega^{m-j+1})^{m-1})u_{1})\\ \\
=u_{1}(1+\omega^{2m-i-j+2})+(\omega^{2m-i-j+2}))^{2}+...+(\omega^{2m-i-j+2}))^{m-1})\\ +u_{2}\omega^{m-i+1}(1+\omega^{2m-i-j+2})+(\omega^{2m-i-j+2}))^{2}+...+(\omega^{2m-i-j+2}))^{m-1})\\ u_{3}(\omega^{m-i+1})^{2}(1+\omega^{2m-i-j+2})+(\omega^{2m-i-j+2}))^{2}+...+(\omega^{2m-i-j+2}))^{m-1})\\ +...+
u_{m}(\omega^{m-i+1})^{m-1}(1+\omega^{2m-i-j+2})+(\omega^{2m-i-j+2}))^{2}+...+(\omega^{2m-i-j+2}))^{m-1}) \\ \\
=(1+\omega^{2m-i-j+2}+(\omega^{2m-i-j+2}))^{2}+...+(\omega^{2m-i-j+2}))^{m-1})\\
(u_{1}+u_{2}\omega^{m-j+1}+u_{3}(\omega^{m-j+1})^{2}+...+u_{m}(\omega^{m-j+1})^{m-1})\\$
\begin{eqnarray*}\omega^{2m-i-j+2}&=&1 \quad \text{if} \quad i+j=2 \ ( mod \ m) \\
    &\neq& 1 \quad \text{otherwise}\end{eqnarray*}
 \begin{eqnarray*}  \therefore  <s_{i},s_{j}>_{u}&=&m(u_{1}+u_{2}\omega^{i-1}+u_{3}(\omega^{i-1})^{2}+...+u_{m}) \quad   \text{if} \ i+j=2 \ ( mod \ m) \\ &=&0 \quad \quad  \text{otherwise}. \end{eqnarray*} Hence the proof.
\end{proof}
\begin{corollary}
$<s_{i},s_{j}>_{u}=0$ for all $i,j \in \{1,2,3,,,,m\}$ if and only if $u=0$.
\end{corollary}
\begin{proof}
If $u=0$, there is nothing to prove. \\
By theorem \ref{r2}, $<s_{i},s_{j}>_{u}=0$ for all $i,j$ other than  $i+j=2 \ ( mod  \ m)$. So it is enough to consider $<s_{i},s_{j}>_{u}=0$ for $i+j=2 \ ( mod \ m)$. Since $m$ is invertible in $R$, $<s_{i},s_{j}>_{u}=0$ for $i+j=2 \ ( mod \ m)$ implies \\
\begin{eqnarray*}
u_{1}+u_{2}+u_{3}+...+u_{m}&=&0 \\
u_{1}+u_{2}\omega+u_{3}(\omega)^{2}+...+u_{m}(\omega)^{m-1}&=&0 \\
u_{1}+u_{2}\omega^{2}+u_{3}(\omega^{2})^{2}+...+u_{m}(\omega^{2})^{m-1}&=&0 \\
............................................................ \\
u_{1}+u_{2}\omega^{m-1}+u_{3}(\omega^{m-1})^{2}+...+u_{m}(\omega^{m-1})^{m-1}&=&0 \\
\end{eqnarray*}
This can be put into the matrix form $\Phi u^{T}=0$.
Since $\Phi^{-1}$ exists, $u^{T}=0$ and therefore $u=0$.
\end{proof}
We write $x\perp y$ if $<x,y>_{u}=0$. We define $(R^{m})^{\perp}= \{x\in R^{m} |<x,y>_{u}=0 \ \text{for all} \  y\in R^{m} \}$. We say that $<x,y>_{u}$ is $nondegenerate$ if $(R^{m})^{\perp}=\{ 0 \}$.
\begin{corollary} \label{r3}
$<x,y>_{u}$ is a $nondegenerate$ bilinear form if $<s_{i},s_{j}>_{u}$ $\in U(R)$ for all $i$ and $j$ such that $i+j=2 \ ( mod \ m)$.
\end{corollary}
\begin{proof}
 $ <x,y>_{u}$=$\sum_{i,j}x_{i}y_{j}<s_{i},s_{j}>_{u}=\sum_{i+j=2mod(m)}x_{i}y_{j}<s_{i},s_{j}>_{u}$ by \ref{r2}. \\
Therefore $<x,y>_{u}=x_{1}y_{1}<s_{1},s_{1}>_{u}+x_{2}y_{m}<s_{2},s_{m}>_{u}+ \\ x_{3}y_{m-1}<s_{3},s_{m-1}>_{u}+....+ x_{m}y_{2}<s_{m},s_{2}>_{u}$ \\
$<x,y>_{u}= 0 \Longrightarrow x_{1}y_{1}<s_{1},s_{1}>_{u}+x_{2}y_{m}<s_{2},s_{m}>_{u}+ \\ x_{3}y_{m-1}<s_{3},s_{m-1}>_{u}+....+ x_{m}y_{2}<s_{m},s_{2}>_{u}=0$ \\
$\Longrightarrow\left(
                  \begin{array}{ccccc}
                    x_{1}<s_{1},s_{1}>_{u} & x_{2}<s_{2},s_{m}>_{u}  & x_{3}<s_{3},s_{m-1}>_{u}  & ... & x_{m}<s_{m},s_{2}>_{u}  \\
                  \end{array}
                \right)\\
 \left(
   \begin{array}{c}
     y_{1} \\
     y_{m} \\
     y_{m1}\\
     . \\
     y_{2} \\
   \end{array}
 \right)
 =0 $.
Since this is true for all $y=(y_{1},y_{2},y_{3},...y_{m})$, \\
$\left(
                  \begin{array}{ccccc}
                    x_{1}<s_{1},s_{1}>_{u} & x_{2}<s_{2},s_{m}>_{u}  & x_{3}<s_{3},s_{m-1}>_{u}  & ..... & x_{m}<s_{m},s_{2}>_{u} \\
                  \end{array}
                \right)$ \\ =(0,0,..,0) \\
$\Longrightarrow x_{1}<s_{1},s_{1}>_{u}=0, x_{2}<s_{2},s_{m}>_{u}=0, x_{3}<s_{3},s_{m-1}>_{u} =0.. \\ x_{m}<s_{m},s_{2}>_{u}=0$ \\
$ x_{1}=x_{2}=x_{3}=...=x_{m}=0$ only when $<s_{i},s_{j}>_{u}\in U(R)$ for $i+j=2 \ (mod \ m)$. \\
Thus $(R^{m})^{\perp}=\{0\}$ and therefore $<x,y>_{u}$ is a nondegenerate bilinear form.

\end{proof}
\begin{corollary}
Let $<x,y>_{u}$ be a nondegenerate bilinear form. Then $M=(<s_{i},s_{j}>)$ is an invertible matrix of order m.
\end{corollary}
\begin{proof}
By theorem \ref{r2}, the matrix M can be written as \\ $M=\left(
                                                        \begin{array}{cccccc}
                                                          <s_{1},s_{1}>_{u} & 0 & 0 & ...&0 & 0 \\
                                                          0 & 0 & 0 & ... & 0 & <s_{2},s_{m}>_{u} \\
                                                          0 & 0 & 0 & ... & <s_{3},s_{m-1}>_{u}& 0 \\
                                                          ... & ... & ... & ... & ... &...\\
                                                          0 & <s_{m},s_{2}>_{u} & 0 & ... & 0 &0\\
                                                        \end{array}
                                                      \right)$ \\
Clearly $D=det \ M = \pm <s_{1},s_{1}>_{u}<s_{2},s_{m}>_{u}<s_{2},s_{m-1}>_{u}.....<s_{m},s_{2}>_{u}$ and the sign is depending on $m$. By corollary \ref{r3}, $D \in U(R)$. So $M^{-1}$ exists and $M^{-1}$ is given by \\
$M^{-1}=\left(
                                                        \begin{array}{cccccc}
                                                          <s_{1},s_{1}>_{u}^{-1} & 0 & 0 & ...&0 & 0 \\
                                                          0 & 0 & 0 & ... & 0 & <s_{m},s_{2}>_{u}^{-1} \\
                                                          0 & 0 & 0 & ... & <s_{m-1},s_{3}>_{u}^{-1}& 0 \\
                                                          ... & ... & ... & ... & ... &...\\
                                                          0 & <s_{2},s_{m}>_{u}^{-1} & 0 & ... & 0 &0\\
                                                        \end{array}
                                                      \right)$.
\end{proof}
\begin{corollary}
Let $R=Z_{n}$ be the ring integers modulo $n$. It is a halidon ring with maximum index $m_{max}=\psi(n)$; where $\psi(n)$ is the halidon function. Then the  number of nondegerate bilinear forms $<x,y>_{u}$ is $\phi(n)^{\psi(n)}$.
\end{corollary}
\begin{proof}
By \ref{r3}, $<x,y>_{u}$ is nondegenrate if and only if $<s_{i},s_{j}>_{u}$ $\in U(R)$ for all $i$ and $j$ such that $i+j=2 \ ( mod \ m)$. Here $|U(R)|=\phi(n)$ and $m=\psi(n)$. Therefore $|<x,y>_{u}|$=$\phi(n)^{\psi(n)}$.
\end{proof}

\begin{theorem}
Let $C=\{ C_{u}| u=(u_{1},u_{2},u_{3},...,u_{m})\in R^{m}\}$ and let G be as in theorem \ref{r1}. Then $C$ is an R-algebra.
\end{theorem}
\begin{proof}
Let $u=(u_{1},u_{2},u_{3},...,u_{m}), v=(v_{1},v_{2},v_{3},...,v_{m})\in R^{m}$ be any two elements in $R$. Since $RG\cong R^{m}$, we can identify the elements $u$ and $v$ as
$\sum_{i=1}^{m}\alpha_{i}g_{i}$ and  $\sum_{i=1}^{m}\beta_{i}g_{i}$ respectively, where $$u_{i}=\sum_{r=1}^{m}\alpha_{m-r+2}(\omega^{i-1})^{r-1}.$$ and $$v_{i}=\sum_{r=1}^{m}\beta_{m-r+2}(\omega^{i-1})^{r-1}.$$ Since $R$ is a halidon ring with index $m$ and $\omega$ is a primitive $m^{th}$ root of unity, the circulant matrix $C_{u}$ can be written as $$C_{u}=\frac{1}{m}\Phi \Lambda_{u}\Phi^{*}, $$ where $$\Lambda_{u}= diag(\lambda_{1},\lambda_{2},\lambda_{3}....,\lambda_{m})$$ such that $$\lambda_{i}=\sum_{r=1}^{m}u_{i}(\omega^{(i-1)})^{(r-1)}$$ and $\Phi^{*}$ is the conjugate transposed of $\Phi$  \cite{pjd}and $\frac{1}{m}\Phi\Phi^{*}=I=\frac{1}{m}\Phi^{*}\Phi$. Clearly $\Lambda_{u}\Lambda_{v}=\Lambda_{uv}.$
\begin{eqnarray*}  \therefore C_{uv}&=& \frac{1}{m}\Phi \Lambda_{uv}\Phi^{*} \\ &=& \frac{1}{m}\Phi \Lambda_{u}\Lambda_{v}\Phi^{*} \\
&=& \frac{1}{m}\Phi \Lambda_{u}\frac{1}{m}\Phi^{*}\Phi\Lambda_{v}\Phi^{*} \\ &=& (\frac{1}{m}\Phi \Lambda_{u}\Phi^{*})(\frac{1}{m}\Phi\Lambda_{v}\Phi^{*}) \\ &=& C_{u}C_{v} \\
\end{eqnarray*}

We define $h:R^{m}\rightarrow C$ by $h(u)=C_{u}$; which is clearly an algebra isomorphism. $\therefore R^{m} \cong C. $ And hence the theorem.
\end{proof}
\begin{theorem}
Let $B=\{ <x,y>_{u}|<x,y>_{u}=xC_{u}y^{T}, \text{for each} \ u\in R^{m}, x,y \in R^{m} \}$. Then $B$ is an $R$-module.
\end{theorem}
\begin{proof}
Let $\alpha \in R$ be any element in R. Then $<x,y>_{u+v}=<x,y>_{u}+ \\ <x,y>_{v}$ and $<x,y>_{\alpha u}= \alpha<x,y>_{u}$. Therefore $B$ is an $R$-module.
\end{proof}

 \section{Discrete Fourier Transforms}
In this section, we deal with the ring of polynomials over a halidon ring which has an application in Discrete Fourier Transforms \cite{jj}. Throughout this section, let $R$ be a finite commutative halidon ring with index $m$ and $R[x]$ denotes the ring of polynomials degree less than $m$ over $R$.
\begin{definition} \cite{jj}
Let $\omega \in R$ be a primitive $m^{th}$ root of unity in $R$ and let $f(x)=\sum \limits_{j=0}^{m-1} f_{j}x^{j} \in R[x]$ with its coefficients vector $(f_{0},f_{1},f_{2},....,f_{m-1}) \in R^{m}$. The \textbf{Discrete Fourier Transform} (DFT) is a map $$ DFT_{\omega}: R[x]\rightarrow R^{m}$$ defined by $$DFT_{\omega}(f(x))=(f_{0}(1),f_{1}(\omega),f_{2}(\omega^{2}),....,f_{m-1}(\omega^{m-1})).$$
\end{definition}\noindent
\begin{remark}
Clearly $DFT_{\omega}$ is a $R$-linear map as $DFT_{\omega}(af(x)+bg(x))=aDFT_{\omega}(f(x))+bDFT_{\omega}(g(x))$ for all $ a, b \in R$. Also, if $R= \mathbb{C}$, the field of complex numbers, then $\omega=cos(\frac{2 \pi}{m})+i sin(\frac{2 \pi}{m})=e^{i\frac{2 \pi}{m}}$ and the Fourier series will become the ordinary series of sin and cos functions.
\end{remark}
\begin{definition} \cite{jj}
The \textbf{convolution} of $f(x)=\sum \limits_{j=0}^{m-1} f_{j}x^{j}$ and $g(x)=\sum \limits_{k=0}^{m-1} g_{k}x^{k}$ in $R[x]$ is defined by $h(x)=f(x)*g(x)= \sum \limits_{l=0}^{m-1} h_{l}x^{l} \in R[x]$ where  \quad $h_{l}=\sum \limits_{j+k=l \ mod \ m} f_{j}g_{k}=\sum \limits_{j=0}^{m-1} f_{j}g_{l-j}$ for $0 \leq l < m$.

\end{definition}
The notion of convolution is equivalent to polynomial multiples in the ring $R[x]/<x^{m}-1>$. The $l^{th}$ coefficient of the product $f(x)g(x)$ is $\sum \limits_{j+k=l \ mod \ m} f_{j}g_{k}$ and hence $$f(x)*g(x)= f(x)g(x) \ mod (x^{m}-1).$$
\begin{proposition} \cite{jj}
For polynomials $f(x),g(x) \in R[x],$ $DFT_{\omega}(f(x)*g(x))=DFT_{\omega}(f(x)).DFT_{\omega}(g(x)),$ where . denotes the pointwise multiplication of vectors.
\end{proposition}
\begin{proof}
$f(x)*g(x)= f(x)g(x) +q(x)(x^{m}-1)$ for some $q(x) \in R[x]$. \\
Replace $x$ by $\omega^{j}$, we get \\
$$f(\omega^{j})*g(\omega^{j})= f(\omega^{j})g(\omega^{j}) +0.$$
$$\therefore \quad \quad \quad DFT_{\omega}(f(x)*g(x))=DFT_{\omega}(f(x)).DFT_{\omega}(g(x)). $$
\end{proof}
\begin{theorem}
For a polynomial $f(x) \in R[x],$ $DFT_{\omega}^{-1}(f(x))= \frac{1}{m}DFT_{\omega^{-1}}(f(x)).$
\end{theorem}
\begin{proof}
The matrix of the transformation $DFT_{\omega}(f(x))$ is $$[DFT_{\omega}(f(x))]=\phi=\left(
                                                                        \begin{array}{ccccc}
                                                                          1 & 1 & 1 & ..... & 1 \\
                                                                          1 & \omega & \omega^{2} & ..... & \omega^{m-1} \\
                                                                          1 & \omega^{2} &(\omega^{2})^{2} & ..... & (\omega^{2})^{m-1}\\
                                                                          . & . & . & ..... & . \\
                                                                          . & . & . & ..... & . \\
                                                                          . & . & . & ..... & . \\
                                                                          1 & \omega^{m-1} & (\omega^{m-1})^{2} & ..... & (\omega^{m-1})^{m-1}\\
                                                                        \end{array}
                                                                      \right)
$$ The matrix $\phi$ is the well known Vandermonde matrix and its inverse is $\frac{1}{m}\phi^{*}$, where $\phi^{*}$ is the matrix transpose conjugated \cite{pjd}. Since $\phi$ is a square matrix and the conjugate of $\omega$ is $\omega^{-1}$, we have  $DFT_{\omega}^{-1}(f(x))= \frac{1}{m}DFT_{\omega^{-1}}(f(x)).$
\end{proof}
\begin{example} \label{rd6}
We know that $R=Z_{49}$ is a halidon ring with index $m=6$ and $\omega=19$. Also, $\omega^{-1}=\omega^{5}=31$. Let $f(x)=2+x+2x^{2}+3x^{3}+5x^{4}+10x^{5} \in R[x].$ Then $DFT_{\omega}(f(x))$ can be expressed as \newline $ \left(
               \begin{array}{c}
                 F_{0} \\
                 F_{1} \\
                 F_{2} \\
                 F_{3} \\
                 F_{4} \\
                 F_{5} \\
               \end{array}
             \right)$ $=$ $\left(
                         \begin{array}{cccccc}
                           1 & 1 & 1 & 1 &1 & 1 \\
                           1 & \omega & \omega^{2} &\omega^{3} & \omega^{4} & \omega^{5} \\
                           1 & \omega^{2} & \omega^{4} &1& \omega^{2} & \omega^{4} \\
                           1 & \omega^{3} & 1 &\omega^{3} & 1 & \omega^{3} \\
                           1 & \omega^{4} & \omega^{2} &1 & \omega^{4} & \omega^{2} \\
                           1 & \omega^{5} & \omega^{4} &\omega^{3} & \omega^{2} & \omega \\
                         \end{array}
                       \right)$  $ \left(
               \begin{array}{c}
                 f_{0} \\
                 f_{1} \\
                 f_{2} \\
                 f_{3} \\
                 f_{4} \\
                 f_{5} \\
               \end{array}
             \right)$
     $\Rightarrow$
     $ \left(
               \begin{array}{c}
                 F_{0} \\
                 F_{1} \\
                 F_{2} \\
                 F_{3} \\
                 F_{4} \\
                 F_{5} \\
               \end{array}
             \right)$ $=$ $ \left(
               \begin{array}{c}
                23  \\
                24 \\
                32  \\
                 44 \\
                 9 \\
                 27 \\
               \end{array}
             \right)$
       $ \left(
               \begin{array}{c}
                 f_{0} \\
                 f_{1} \\
                 f_{2} \\
                 f_{3} \\
                 f_{4} \\
                 f_{5} \\
               \end{array}
             \right)$ $=6^{-1}$$\left(
                         \begin{array}{cccccc}
                           1 & 1 & 1 & 1 &1 & 1 \\
                           1 & \omega^{5} & \omega^{4} &\omega^{3} & \omega^{2} & \omega \\
                           1 & \omega^{4} & \omega^{2} &1& \omega^{4} & \omega^{2} \\
                           1 & \omega^{3} & 1 &\omega^{3} & 1 & \omega^{3} \\
                           1 & \omega^{2} & \omega^{4} &1 & \omega^{2} & \omega^{4} \\
                           1 & \omega & \omega^{2} &\omega^{3} & \omega^{4} & \omega^{5} \\
                         \end{array}
                       \right)$  $ \left(
               \begin{array}{c}
                 F_{0} \\
                 F_{1} \\
                 F_{2} \\
                 F_{3} \\
                 F_{4} \\
                 F_{5} \\
               \end{array}
             \right)$
             \newline
      $\Rightarrow$     $ \left(
               \begin{array}{c}
                 f_{0} \\
                 f_{1} \\
                 f_{2} \\
                 f_{3} \\
                 f_{4} \\
                 f_{5} \\
               \end{array}
             \right)$ $=41$$\left(
                         \begin{array}{cccccc}
                           1 & 1 & 1 & 1 &1 & 1 \\
                           1 & 31 & 30 &48 & 18 & 19 \\
                           1 & 30 & 18 &1& 30 & 18 \\
                           1 & 48 & 1 &48 & 1 & 48 \\
                           1 & 18 & 30 &1 & 18 & 30 \\
                           1 & 19& 18 &48 & 30 & 31 \\
                         \end{array}
                       \right)$  $ \left(
               \begin{array}{c}
                 F_{0} \\
                 F_{1} \\
                 F_{2} \\
                 F_{3} \\
                 F_{4} \\
                 F_{5} \\
               \end{array}
             \right)$     \newline
If $ \left(
               \begin{array}{c}
                 F_{0} \\
                 F_{1} \\
                 F_{2} \\
                 F_{3} \\
                 F_{4} \\
                 F_{5} \\
               \end{array}
             \right)$     $=$ $ \left(
               \begin{array}{c}
                 23 \\
                 24 \\
                 32 \\
                 44 \\
                 9 \\
                 27 \\
               \end{array}
             \right)$, then a direct calculation gives $ \left(
               \begin{array}{c}
                 f_{0} \\
                 f_{1} \\
                 f_{2} \\
                 f_{3} \\
                 f_{4} \\
                 f_{5} \\
               \end{array}
             \right)$ $=$  $ \left(
               \begin{array}{c}
                2 \\
                1 \\
                 2 \\
                 3 \\
                 5 \\
                 10 \\
               \end{array}
             \right)$ \newline as expected.

\end{example}
The programme-6 and programme-7 will enable us to calculate Discrete Fourier Transform and its inverse. We can cross-check the programmes against example \ref{rd6}. \\

If $R=Z_{100001}$, $m=10$, $\omega=26364$ and $f(x)=1+2x+3x^{2}+4x^{3}+5x^{4}+6x^{5}+7x^{6}+8x^{7}+9x^{8}+x^{9} \in R[x]$, then $ \left( \begin{array}{c}
                 F_{0} \\
                 F_{1} \\
                 F_{2} \\
                 F_{3} \\
                 F_{4} \\
                 F_{5} \\
                 F_{6} \\
                 F_{7} \\
                 F_{8} \\
                 F_{9} \\
               \end{array}
             \right)$     $=$ $ \left(
               \begin{array}{c}
                 46 \\
                 19019 \\
                 3314 \\
                 10082 \\
                 48017 \\
                 4 \\
                80347 \\
                 18172 \\
                 68413 \\
                 52627 \\
               \end{array}
             \right)$. \\
              Also, we can verify the inverse DFT using the above data. \\
\section{Conclusions}
The halidon rings are useful to extend the two famous theorems of Graham Higman(1940) and Maschke(1899). Since Higman's theorem and Maschke's theorem have a wide range of applications in group rings, group algebras and representation theory, there is a big scope of wider applications of halidon rings. The field of complex numbers is an infinite halidon ring with any index $m>0 $. The field of real numbers is an infinite halidon ring with index 2 and the ring integers is a trivial halidon ring with index m=1. Using the field of complex numbers, we can create infinite halidon rings of square matrices with any index $m>0$. This will open new vistas of applications in algebra and number theory. Another area of application is the coding theory on which the author is currently working with.  \\
\textbf{Acknowledgment} I am very much indebted to Prof. M.I.Jinnah, Former Head of Mathematics, University of Kerala, India, for his constructive suggestions and support.

\newpage
\section{Appendix}

\begin{verbatim}
Programme-1 : To check whether Z(n) is a trivial or nontrivial halidon ring
#include <iostream>
#include <cmath>
using namespace std;
int main() {
	cout << "To check whether Z(n) is a trivial or nontrivial halidon ring." << endl;
	unsigned long long int t = 0, n = 1, w = 1, hcf, hcf1,
		d = 1, k = 1, q = 1, p = 1, b=0, c=0, temp = 1;
	cout << "Enter an integer n >0: ";
	cin >> n;
	if (n % 2 == 0) {
		cout << "Z(" << n << ") is a trivial halidon ring." << endl;
	}
		for (w = 1; w < n; ++w) {
		for (int i = 1; i <= n; ++i) {
			if (w % i == 0 && n % i == 0) {
				hcf = i;

		}
		} if (hcf == 1) {
			++t; // cout << "  " << w << "  ";
		}
	}
	for (w = 1; w < n; ++w) {
		for (int i = 1; i <= n; ++i) {
			if (w % i == 0 && n % i == 0) {
				hcf = i;

			}
		}
		if (hcf == 1) {
		for (int k = 1; k <= t; ++k) {
		q = q * w; q = q % n;
		if (q == 1) { if (temp <= k) { temp = k; } break; }
		}
		}
	}
	
	for (w = 2; w < n; ++w) {
		for (int i = 1; i <= n; ++i) {
		if (w % i == 0 && n % i == 0) {
		hcf = i;
			}
		}
		if (hcf == 1) {
		for (int k = 1; k <= temp; ++k) {
		q = q * w; q = q % n; if (q == 1) {
		for (int i = 1; i <= n; ++i) {
		if (k % i == 0 && n % i == 0) {
		hcf1 = i;
		}
		}
		if (hcf1 == 1) {
		for (int j = 1; j < k; ++j)
		{
		if (k%j == 0) {
		d = j;
		for (int l = 1; l <= d; ++l)
		{
		p = (p*w); p = p % n;
		}
		for (int i = 1; i <= n; ++i) {
		if ((p - 1) % i == 0 && n % i == 0) {
		hcf = i;
		}
		}
             if (hcf == 1) {

		p = 1; b = b + 1;
		}
		else p = 1;
		c = c + 1;
     		}
		}
		if (c == b) { cout << "   Z(" << n << ")" <<
							
" is a halidon ring with index m= " << k <<

" and w= " << w << "."<< endl; } {p = 1; c = 0; b = 0; }

		break;
		}
		}
		}
		}
	} return 0;
}

Programme-2: To check whether an element in ZnG; G is a cyclic group of
order m has a multiplicative inverse or not
#include<iostream>
#include<cmath>
using namespace std;
int main() {
	cout << "To check whether an element in Z(n)G;" <<
		"G is a cyclic group of order m" <<
              “has a multiplicative inverse or not” << endl;
	int a[100], b[100], c[100], d[100], e[100], w1[100], m = 1,
		t = 0, x = 1, s = 0, s1 = 0, l = 0, m1 = 1, hcf = 1,
		n = 1, i = 1, k = 0, q = 1, p = 1, r = 1, w = 1;
	cout << "Enter n =" << endl;
	cin >> n;
	cout << "Enter index m =" << endl;
	cin >> m;
	cout << "Enter  m^(-1) =" << endl;
	cin >> m1;
	cout << "Enter primitive m th root w =" << endl;
	cin >> w;
	for (i = 0; i < m; ++i)
    { w1[i] =( (long long int)pow(w, i)) % n;
	cout << "w1[" << i << "]" << w1[i] << endl; }
		for (int i = 1; i < m + 1; ++i) {
			cout << "Enter a["<<i<<"]=" << endl;
			cin >> a[i];
		}
	a[0] = a[m];
		for (int r = 1; r < m + 1; ++r) {
		for (int j = 1; j < m + 1; ++j)
		{
			l = (m - j + 2) % m;
			x = ((j - 1) * (r - 1)) % m;
			k = k + (a[l] * w1[x]) % n; k = k % n;
			// cout << "k=" << k << endl;
		} c[r] = k; cout << "c[" << r << "]=" << c[r] << endl;
		k = 0;
	}
	for (r = 1; r < m+1; ++r) {
		for (int i = 1; i <= n; ++i) {
			if (c[r] % i == 0 && n % i == 0) {
				hcf = i;
			}
		}
		if (hcf == 1) {
			cout << "c[" << r << "] is a unit" << endl;
		}
		else { cout << "c[" << r <<
		"] is a not unit. So there is no multiplicative inverse." <<
			endl; t = 1; }
	}
	for (r = 1; r < m + 1; ++r) {
		for (int i = 1; i <= n; ++i) {
			e[r] = (c[r] * i) % n;
			if (e[r] == 1) {
				b[r] = i;
	cout << " The inverse of c[" << r << "] is " << b[r] << endl;
			}
		}
	}
	b[0] = b[m];
	for (int r = 1; r < m + 1; ++r) {
		for (int j = 1; j < m + 1; ++j)
		{
	l = (m - j + 2) % m;
	x = (m*m-(j - 1) * (r - 1)) % m;  cout << "x= " << x << endl;
	k = k + (m1*b[l] * w1[x]) % n; k = k % n;
       cout << "k=" << k << endl;
	}
d[r] = k; cout << "d[" << r << "]=" << d[r] << endl;
	k = 0;
	}
	if (t == 1) {
		s = m;
	mylabel2:
		cout << a[m - s + 1] << "g^(" << m - s << ") + ";
		s--;
		if (s > 0) goto mylabel2; cout <<
       "has no multiplicative inverse." << endl;
	}
	else {cout << "The inverse of ";
	s = m;
mylabel:
	cout << a[m - s + 1] << "g^(" << m - s << ") + ";
	s--;
	if (s > 0) goto mylabel; cout << "is" << endl;
	s1 = m;
mylabel1:
	cout << d[m - s1 + 1] << "g^(" << m - s1 << ") + ";
	s1--;
	if (s1 > 0) goto mylabel1; cout << "." << endl; }
	return 0;
	}

Programme-3: To find idempotents, involutions and units in Zn
 #include<iostream>
#include<cmath>
using namespace std;
int main() {
cout << "To find idempotents, involutions and units in Z(n)." << endl;
	long long int i, j=0,n, k=0, x, y;
	cout << "Enter n=" << endl;
	cin >> n;
	cout << "The involutions are ";
	for (i = 1; i < n + 1; ++i) {
		x = (i * i) % n;
		if (x == 1) { cout << i << "  "; }
	} cout << "." << endl;
	cout << "The idempotents are ";
		for (i = 0; i < n + 1; ++i) {
			x = (i * i) % n;
			if (x == i) { cout << i << "  "; j++; }
		} cout << "." << endl; cout <<
         "Number of idempotents = " << j << endl;
		cout << "The units and its inverse are  "<<endl;
		for (y = 1; y < n + 1; ++y) {
			for (i = 1; i < n + 1; ++i) {
				x = (y * i) % n;
	   if (x == 1) { cout << y << " , " << i << endl; k++; }
			}
		} cout << "Number of units = " << k;
		return 0;
}

Programme-4: To find the inverse of an element in Z(n)G;G is a cyclic group of
order m through lamda units

 #include<iostream>
#include<cmath>
using namespace std;
int main() {
	cout << "To find the inverse of an element in Z(n)G;" <<
		"G is a cyclic group of order m through lamda units." << endl;
	int a[100], b[100], l[100], l1[100], w1[100],
	m = 1, t = 0, x = 1, y=1,s = 0, s1 = 0, m1 = 1, hcf = 1, n = 1,
	i = 1, k = 0, q = 1, p = 1, r = 1, w = 1;
		cout << "Enter n =" << endl;
	cin >> n;
	cout << "Enter index m =" << endl;
	cin >> m;
	cout << "Enter  m^(-1) =" << endl;
	cin >> m1;
	cout << "Enter primitive m th root w =" << endl;
	cin >> w;
	for (i = 0; i < m; ++i) { w1[i] = ((long long int)pow(w, i)) % n;
	cout << "w1[" << i << "]=" << w1[i] << endl; }
	cout << "Enter lamda values which have inverse" << endl;
	for (int i = 1; i < m + 1; ++i) {
		cout<< "l[" << i << "]=" << endl;
		cin >> l[i];
	}
	cout << "Enter lamda inverse values " << endl;
	for (int i = 1; i < m + 1; ++i) {
		cout << "l1[" << i << "]=" << endl;
		cin >> l1[i];
	}
		for (int r = 1; r < m + 1; ++r) {
		for (int j = 1; j < m + 1; ++j)
		{
			x = ((j - 1) * (r - 1)) % m;
			k = k + (m1*l[j] * w1[x]) % n; k = k % n;
			// cout << "k=" << k << endl;
		} a[r] = k; cout << "a[" << r << "]=" << a[r] << endl;
		k = 0;
	}
	for (int r = 1; r < m + 1; ++r) {
		for (int j = 1; j < m + 1; ++j)
		{
			x = ((j - 1) * (r - 1)) % m;  cout << "x= " << x << endl;
			k = k + (m1 * l1[j] * w1[x]) % n; k = k % n;
			cout << "k=" << k << endl;
		} b[r] = k; cout << "b[" << r << "]=" << b[r] << endl;

		k = 0;
	}
		cout << "The inverse of a= ";
		s = 1;
	mylabel:
		cout << a[s ] << "g^(" << s-1 << ") + ";
		s++;
		if (s <m+1) goto mylabel; cout << endl;  cout << "is b=";
		s1 = m;
	mylabel1:
		cout << b[m - s1 + 1] << "g^(" << m - s1 << ") + ";
		s1--;
		if (s1 > 0) goto mylabel1; cout << "." << endl;
		cout << "Note: Please neglect the last + as" <<
        "it is unavoidable for a for loop.";
	return 0;
}

Programme-5: To find the idempotent elements in ZnG;G is a cyclic group of order m through
lamda takes idempotent values in Zn
\begin{verbatim}
#include<iostream>
#include<cmath>
using namespace std;
int main() {
	cout << "To find the idempotent elements in Z(n)G;"<<
		"G is a cyclic group of order m through" <<
		"lamda takes idempotent values in Z(n)." << endl;
	int a[100], l[100], w1[100], m = 1, t = 0, x = 1, y = 1, s = 0,
	s1 = 0, m1 = 1, n = 1, i = 1, k = 0,  r = 1, w = 1;

	cout << "Enter n =" << endl;
	cin >> n;
	cout << "Enter index m =" << endl;
	cin >> m;
	cout << "Enter m^(-1) =" << endl;
	cin >> m1;
	cout << "Enter primitive m th root w =" << endl;
	cin >> w;
	for (i = 0; i < m; ++i) { w1[i] = ((long long int)pow(w, i)) % n;
	cout << "w1[" << i << "]=" << w1[i] << endl; }
	cout << "Enter lamda values which are idempotents" << endl;
	for (int i = 1; i < m + 1; ++i) {
	cout << "l[" << i << "]=" << endl;
		cin >> l[i];
	}
		for (int r = 1; r < m + 1; ++r) {
		for (int j = 1; j < m + 1; ++j)
		{
			x = ((j - 1) * (r - 1)) % m;
			k = k + (m1 * l[j] * w1[x]) % n; k = k % n;
			}
      a[r] = k; cout << "a[" << r << "]=" << a[r] << endl;
		k = 0;
	}
	cout << "The idempotent element in RG is  e= ";
	s = 1;
mylabel:
	cout << a[s] << "g^(" << s - 1 << ") + ";
	s++;
	if (s < m + 1) goto mylabel; cout << endl;
	cout << endl;
cout << "Note: Please neglect the last + as it is" <<
"unavoidable for a for loop.";
	return 0;
}

Programme-6: Discrete Fourier Transform

#include <iostream>
#include<cmath>
using namespace std;
int main()
{
cout << "Discrete Fourier Transform" << endl;
unsigned long long int a[200][200], b[200][200],
mult[200][200],	q=1,m=1, n=1, w2=1,w=1, r1, c1, r2,
c2, i, j, k, t=1;
cout << "Enter n,m,w: ";
	cin >> n >> m >> w;
	r1 = m; c1=m;
	r2 = m; c2=1;
	for (i = 0; i < r1; ++i)
	for (j = 0; j < c1; ++j)
		{
			 t = (i*j)%m;
	 if (t == 0) a[i][j] = 1;
	 else
	for (q = 1; q < t + 1; ++q) { w2 = (w2 * w) % n; }
			 a[i][j] = w2; w2 = 1;
		}
	for (i = 0; i < r1; ++i)
		for (j = 0; j < c1; ++j)
	   {
		cout<<"  a"<<i+1<<" "<<j+1<<"="<<	a[i][j] ;
		if (j == c1 - 1)
			cout << endl;
		}
	cout << endl << "Enter coefficient vector of
    the polynomial:" << endl;
	for (i = 0; i < r2; ++i)
	for (j = 0; j < c2; ++j)
		{
    cout << "Enter element f" << i << " = ";
	        cin >> b[i][j];
		}
	for (i = 0; i < r1; ++i)
	for (j = 0; j < c2; ++j)
		{
    		mult[i][j] = 0;
		}
	for (i = 0; i < r1; ++i)
	for (j = 0; j < c2; ++j)
	for (k = 0; k < c1; ++k)
		{
    		mult[i][j] += (a[i][k]) * (b[k][j]);
		}
	cout << endl << "DFT Output: " << endl;
	for (i = 0; i < r1; ++i)
	for (j = 0; j < c2; ++j)
		{
			cout << "F"<< i << "="<< mult[i][j]%n;
	if (j == c2 - 1)
			cout << endl;
		}
	return 0;
        }

Programme-7: Inverse Discrete Fourier Transform

#include <iostream>
#include<cmath>
using namespace std;
int main()
{
	cout << "Inverse Discrete Fourier Transform" << endl;
	unsigned long long int a[100][100], b[100][100],
    mult[100][100], p=1, q=1, l=1, m = 1, m1 = 1, w1 = 1,
    w2=1, n = 1, w = 1, r1, c1, r2, c2, i, j, k,
    c=1,t = 1;
	cout << "Enter n,m, w: ";
	cin >> n >> m >> w;
	for (l = 1; l < n; ++l)
        {
		c = (l * m) % n;
		if (c == 1)
        {
			m1 = l;
		}
	    }
	for (p = 1; p < m; ++p)
    	{
		w1 = (w1 * w) % n;
	    }
	r1 = m; c1 = m;
	r2 = m; c2 = 1;
	for (i = 0; i < r1; ++i)
	for (j = 0; j < c1; ++j)
		{
			t = (i * j) % m;
	if (t == 0) a[i][j] = 1;
	else
	for (q = 1; q < t + 1; ++q) { w2 = (w2 * w1) % n; }
			a[i][j] = w2; w2 = 1;
		}
	for (i = 0; i < r1; ++i)
	for (j = 0; j < c1; ++j)
		{
	cout << "  a" << i + 1 << j + 1 << "=" << a[i][j];
	if (j == c1 - 1)
	cout << endl;
		}
	cout << endl << "Enter DFT vector :" << endl;
	for (i = 0; i < r2; ++i)
	for (j = 0; j < c2; ++j)
		{
	cout << "Enter element F" << i << " = ";
	cin >> b[i][j];
		}
	for (i = 0; i < r1; ++i)
	for (j = 0; j < c2; ++j)
		{
	mult[i][j] = 0;
		}
	for (i = 0; i < r1; ++i)
	for (j = 0; j < c2; ++j)
	for (k = 0; k < c1; ++k)
		{
	mult[i][j] += (a[i][k]) * (m1 * b[k][j]);
   		}
	cout << endl << "Polynomial vector: " << endl;
	for (i = 0; i < r1; ++i)
	for (j = 0; j < c2; ++j)
		{
	cout << "f" << i << "=" << mult[i][j] % n;
	if (j == c2 - 1)
	cout << endl;
		}
	return 0;
        }


\end{verbatim}

\end{document}